\documentclass[12pt,reqno]{amsart}
\usepackage{epsfig}
\usepackage{graphicx}
\usepackage{mathrsfs}
\usepackage{array}
\usepackage{verbatim}
\usepackage{amssymb}
\usepackage{latexsym}
\usepackage{amsmath}
\usepackage{amsfonts}
\usepackage{amsbsy}
\usepackage{amscd}
\usepackage[usenames]{color} %\color{Red}
\newcommand{\R}{{\mathbb R}}
\newcommand{\Z}{{\mathbb Z}}

\newtheorem{prop}{Proposition}[section]
\newtheorem{lem}[prop]{Lemma}

\newtheorem{cor}[prop]{Corollary}
\newtheorem{thm}[prop]{Theorem}

\newtheorem{rem}[prop]{Remark}

\begin{document}
\baselineskip=17pt

\title[]
{Spectrality of convolution and product measures: a block decomposition approach}

\author{ting zhou}

\address{School of Mathematicsand Statistics, Central South University, Changsha 410085, China}
\email{tingzhouhn@163.com}

\date{\today}
\keywords{Spectral measure, convolution measure, product measure, Fourier transform, spectral decomposition}
\subjclass[2010]{Secondary 42C05; Primary 28A80.}

\begin{abstract}
We study the spectrality of convolution and product measures.
Let~$\mu$~be a compactly supported Borel probability measure on~$\Bbb{R}^d$~such that
~$\mathcal{Z}(\widehat\mu)\cup\{0\}=T$~is an additive subgroup, and let~$\nu$~be a Borel probability measure on~$\Bbb{R}^d$.
We prove that~$\mu*\nu$~is spectral if and only if~$T$~is a spectrum of~$\mu$, ~$\nu$~is spectral, and~$supp\nu$~is contained in a translate of the annihilator of~$T$.
Moreover, every spectrum of~$\mu*\nu$~decomposes into full cosets of~$T$, and every choice of one point from each coset is a spectrum of~$\nu$.
We also obtain a convolution-factor criterion which strengthens a result of Wu.
As applications, we characterize the spectrality of~$L_{[0,1]}*\nu$~and prove that the spectrality of~$L_{[0,1]^d}*\nu$~implies that of~$\nu$~without a compact support assumption.
For product measures, we give necessary and sufficient
conditions for an orthogonal set to be a spectrum
by relating its block structure to spectra of the factors.
No compact support assumptions are required.
\end{abstract}
\maketitle
\section{\bf Introduction\label{sect.1}}
\indent
Let~$\mu$~be a Borel probability measure on~$\Bbb{R}^d$. The Fourier transform of~$\mu$~is defined by
$$\widehat\mu(\xi)=\int_{\Bbb{R}^d}e^{-2\pi i\langle x,\xi\rangle}\,d\mu(x),\qquad \xi\in\Bbb{R}^d.$$
Denote by~$\mathcal{Z}(\widehat\mu)=\{\xi\in\Bbb{R}^d:\widehat\mu(\xi)=0\}$~the zero set of~$\widehat\mu$.
For~$\lambda\in\Bbb{R}^d$, write~$e_\lambda(x)=e^{2\pi i\langle\lambda,x\rangle}$.
A countable set~$\Lambda\subset\Bbb{R}^d$~is called an \emph{orthogonal set} for~$\mu$~if
~$E(\Lambda)=\{e_\lambda:\lambda\in\Lambda\}$~is an orthonormal family in~$L^2(\mu)$.
Equivalently,
\begin{equation}\label{(1.1)}
(\Lambda-\Lambda)\setminus\{0\}\subset\mathcal{Z}(\widehat\mu).
\end{equation}
If~$E(\Lambda)$~is an orthonormal basis for~$L^2(\mu)$, then~$\mu$~is called a \emph{spectral measure}, and~$\Lambda$~is called a \emph{spectrum} of~$\mu$.
For a measurable set~$\Omega\subset\Bbb{R}^d$~with positive finite Lebesgue measure, we write
~$L_\Omega=|\Omega|^{-1}\mathbf1_\Omega\,dx$~for its normalized Lebesgue measure.
The set~$\Omega$~is called spectral if~$L_\Omega$~is a spectral measure.\\
\indent
The study of spectral sets began with Fuglede~\cite{F74} in 1974 in connection with commuting self-adjoint partial differential operators.
He proposed the celebrated Fuglede conjecture: a measurable set of positive finite Lebesgue measure is spectral if and only if it tiles the Euclidean space by translations.
Tao~\cite{T04} disproved the conjecture in dimensions at least five, and subsequent work of Kolountzakis and Matolcsi~\cite{KM06a,KM06b} gave counterexamples in dimensions at least three.
For convex domains, Lev and Matolcsi~\cite{LM22} proved the conjecture in every dimension.
A recent preprint of Zhang~\cite{Z26} reports counterexamples to both implications in dimension two.
Although the conjecture fails in general, the relation between spectrality and translational tilings continues to motivate research in harmonic analysis, geometry and number theory.\\
\indent
After the Fuglede conjecture was disproved, research focus gradually shifted from
spectrality of sets to spectrality of more general measures. Fractal measures
(such as self-similar measures, self-affine measures, and Moran measures), due
to their singularity, non-integer dimensions, and complex zero set structures,
have become important objects in spectral measure theory. The study of spectrality
of fractal measures began with the pioneering work of Jorgensen and Pedersen
\cite{JP98} in 1998. They considered the self-similar measures
$\mu_{b,a}$ generated by the iterated function system
$\tau_0(x) = x/b$, $\tau_1(x) = (x + a)/b$ ($b > 1$ an integer, $a \in \mathbb{Z}$).
Surprisingly, they discovered that for $b = 3$, $a = 2$ (the classical middle-third
Cantor measure), $\mu_{3,2}$ is \textbf{not} a spectral measure, while for
$b = 4$, $a = 2$ (the middle-fourth Cantor measure), $\mu_{4,2}$ \textbf{is} a
spectral measure. This finding revealed the high sensitivity of spectrality to
the arithmetic structure of digit sets and expansion factors, this provided the first examples of singular continuous spectral measures. This result initiated the systematic study of fractal
spectral measure theory.
\\
\indent
Following the work of Jorgensen and Pedersen, considerable
progress has been made in the study of spectral self-similar
measures. \L{}aba and Wang~\cite{LW02} established spectrality
for a class of Cantor measures. Hu and Lau~\cite{HuLau2008}
studied the existence of infinite orthogonal families of
exponential functions for Bernoulli convolutions, while
Dai~\cite{Dai12} characterized their spectrality.
Dutkay, Han and Sun~\cite{Dutkay2009} investigated the spectra
of the one-fourth Cantor measure using a tree structure.
More generally, Dutkay, Haussermann and Lai~\cite{DHL19}
proved that Hadamard triples generate self-affine spectral
measures.
\\
\indent
In spectral measure theory, an important issue is to investigate how spectrality is preserved or lost under algebraic operations on measures. We first consider convolution. For Borel probability measures~$\mu$~and~$\nu$~on~$\Bbb{R}^d$, one has
\begin{equation}\label{(1.2)}
\widehat{\mu*\nu}(\xi)=\widehat\mu(\xi)\widehat\nu(\xi),\qquad
\mathcal{Z}(\widehat{\mu*\nu})=\mathcal{Z}(\widehat\mu)\cup\mathcal{Z}(\widehat\nu).
\end{equation}
Thus the Fourier zero set of a convolution is determined by the zero sets of its factors.
This raises the following question.

\noindent\textbf{(P1)}~Under what conditions does the spectrality of~$\mu*\nu$~imply that both~$\mu$~and~$\nu$~are spectral?

\indent
Gabardo and Lai~\cite{GL14} proved that both factors are spectral when their convolution is Lebesgue measure on a cube, up to normalization.
Wu~\cite[Theorem~1.2]{W24} considered a condition on the differences of Fourier zeros.
In the notation~$\eta=\mu*\nu$~and~$D=\mathcal{Z}(\widehat\mu)\setminus\mathcal{Z}(\widehat\nu)$, this condition is
\begin{equation}\label{(1.3)}
(D-D)\cap\mathcal{Z}(\widehat\eta)\subset D.
\end{equation}
For compactly supported probability measures on~$\R$, Wu~\cite{W24} proved that~$\eta$~being spectral implies that~$\mu$~is spectral.
If, in addition,~$\nu$~has finite support and~$\widehat\nu$~is periodic, then~$\nu$~is also spectral.
In this paper, we obtain spectrality of both factors under~\eqref{(1.3)}~without these additional assumptions on~$\nu$.
Our result holds on~$\Bbb{R}^d$, and only~$\mu$~is required to have compact support.
\\
\indent
On the other hand, even if both~$\mu$~and~$\nu$~are spectral,
their convolution~$\mu*\nu$~need not be spectral.
For example,~$L_{[0,1]}$~is a spectral measure with spectrum~$\Z$,
whereas~$L_{[0,1]}*L_{[0,1]}$~is not spectral by
Corollary~\ref{cor2.7} below.
Thus spectrality of the two factors alone is not sufficient.
This leads us to ask what additional conditions on their supports,
together with spectrality of the factors, ensure that the convolution
is spectral.
\\
\indent
We next consider product measures. If~$\mu$~and~$\nu$~are spectral with spectra~$\Lambda_\mu$~and~$\Lambda_\nu$, respectively, then~$\Lambda_\mu\times\Lambda_\nu$~is a spectrum of~$\mu\times\nu$.
The converse leads to the product spectrality problem.

\noindent\textbf{(P2)}~If~$\mu\times\nu$~is spectral, must both~$\mu$~and~$\nu$~be spectral?

\indent
For domains, Greenfeld and Lev~\cite{GL16} proved a positive result when one factor is an interval.
Kolountzakis~\cite{K16} gave another proof and treated unions of two intervals, while Greenfeld and Lev~\cite{GL20} considered convex polygonal factors.
Somlai~\cite{S26} constructed a counterexample to the general product assertion.
Taking normalized Lebesgue measures on those sets also gives a counterexample to~\textbf{(P2)}.
Positive results remain available under additional geometric assumptions; for example, Ramabadran and van Vliet~\cite{RV25} treated a perturbed interval factor.
\\
\indent
For general measures, Kolountzakis, Lai, Lai and Li~\cite{KLLL26}
studied the case in which the Fourier zero set of one factor
is contained in an additive subgroup.
In particular, their Theorem~1.2(2) gives a positive answer
to~\textbf{(P2)}~when~$\mathcal{Z}(\widehat\mu)\cup\{0\}=T$
is an additive subgroup.
We give a different proof by decomposing an arbitrary spectrum
of~$\mu\times\nu$~according to the cosets of~$T$~in its
first-coordinate projection.
We show that every nonempty projected block is a full coset of~$T$
and hence a spectrum of~$\mu$.
Moreover, choosing one first-coordinate frequency from each block
and taking the union of the corresponding second-coordinate
sections always produces a spectrum of~$\nu$.
These conclusions hold without compact support assumptions
on either factor.
\\
\indent
We now describe our main results.
We first study the spectrality of convolution measures.
For the convolution theorem, assume that
\begin{equation}\label{(1.4)}
\mathcal{Z}(\widehat\mu)\cup\{0\}=T,~~~~T\subset\Bbb{R}^d\text{ an additive subgroup}.
\end{equation}
Since~$\widehat\mu(0)=1$~and~$\widehat\mu$~is continuous,~$T$~has no nonzero elements sufficiently close to the origin. Hence~$T$~is discrete and countable.
Under this subgroup assumption, we obtain the following
characterization of spectral convolution measures.
\begin{thm}\label{thm1.1}
Let~$\mu$~be a compactly supported Borel probability measure
on~$\Bbb{R}^d$, and let~$\nu$~be a Borel probability measure
on~$\Bbb{R}^d$.
Suppose that
$$
\mathcal{Z}(\widehat{\mu})\cup\{0\}=T
$$
for an additive subgroup~$T\subset\Bbb{R}^d$, and define
$$
T^\perp
=\{x\in\Bbb{R}^d:\langle t,x\rangle\in\Bbb{Z}
\text{ for all }t\in T\}.
$$
Then~$\mu*\nu$~is spectral if and only if~$T$~is a spectrum
of~$\mu$,~$\nu$~is spectral, and
$$
supp\nu\subset a+T^\perp
$$
for some~$a\in\Bbb{R}^d$.\\
\indent
More precisely, for any spectrum~$\Lambda$~of~$\mu*\nu$,
each nonempty block~$\Lambda\cap(c+T)$~equals~$c+T$,
and every choice of one point from each block is a spectrum
of~$\nu$.
Conversely, if these equivalent conditions hold, then
~$\Gamma+T$~is a spectrum of~$\mu*\nu$~for every spectrum
~$\Gamma$~of~$\nu$.
\end{thm}

The proof also yields a criterion for decomposing a convolution
spectrum into spectra of its factors.
As an application, we strengthen the convolution-factor theorem
of Wu~\cite{W24}.
Let~$\mu$~be compactly supported and let~$\nu$~be an arbitrary
Borel probability measure on~$\Bbb{R}^d$.
Writing
$$
D=\mathcal{Z}(\widehat{\mu})
  \setminus\mathcal{Z}(\widehat{\nu}),
$$
we show that, under the condition
$$
(D-D)\cap\mathcal{Z}(\widehat{\mu*\nu})\subset D,
$$
the spectrality of~$\mu*\nu$~implies the spectrality of both
~$\mu$~and~$\nu$.
Thus the finite-support and Fourier-periodicity assumptions
on the second factor in Wu's result are not required.
\\
\indent
In particular, taking~$\mu=L_{[0,1]}$~and~$T=\Bbb{Z}$
gives a complete characterization in one dimension:
$$L_{[0,1]}*\nu\text{ is spectral}\quad\Longleftrightarrow\quad
\nu\text{ is spectral and }supp\nu\subset a+\mathcal{Z}\text{ for some }a\in\R.$$
In this case, the result of \L{}aba and Wang~\cite[Theorem~1.2(2)]{LW06} further implies that~$\nu$~is a uniform probability measure on a finite subset of~$a+\mathcal{Z}$.
\\
\indent
For higher-dimensional cubes, the lattice-support
characterization above does not extend directly.
Nevertheless, the implication from the spectrality of the
convolution to that of the second factor remains valid.

\begin{thm}\label{thm1.2}
Let~$\nu$~be a Borel probability measure on~$\Bbb{R}^d$.
If~$L_{[0,1]^d}*\nu$~is spectral, then~$\nu$~is spectral.
\end{thm}

For product measures, our general block characterization
yields the following result. It recovers the factor
spectrality conclusion of~\cite[Theorem~1.2(2)]{KLLL26}
and describes the structure of every spectrum of the
product measure.

\begin{thm}\label{thm1.3}
Let~$\mu$~and~$\nu$~be Borel probability measures
on~$\Bbb{R}^{d_1}$~and~$\Bbb{R}^{d_2}$, respectively.
Suppose that
$$
\mathcal{Z}(\widehat{\mu})\cup\{0\}=T
$$
for an additive subgroup~$T\subset\Bbb{R}^{d_1}$.
Then~$\mu\times\nu$~is spectral if and only if~$T$~is
a spectrum of~$\mu$~and~$\nu$~is spectral.
\\
\indent
More precisely, let~$\Lambda$~be any spectrum of~$\mu\times\nu$,
and let~$\pi_1(x,y)=x$~be the projection onto~$\Bbb{R}^{d_1}$.
Let~$\{c_i+T:i\in I\}$~be the distinct cosets of~$T$
meeting~$\pi_1(\Lambda)$, and define
$$
A_i=\pi_1(\Lambda)\cap(c_i+T),\qquad
B_{i,\lambda}
=\{\eta\in\Bbb{R}^{d_2}:(\lambda,\eta)\in\Lambda\}
\quad(\lambda\in A_i).
$$
Then~$A_i=c_i+T$~for every~$i\in I$.
Moreover, for every choice of~$\lambda_i\in A_i$,
the sets~$B_{i,\lambda_i}$~are pairwise disjoint,
and their union is a spectrum of~$\nu$.
\end{thm}

Neither factor in Theorem~\ref{thm1.3} is required to have
compact support.
Applying this theorem successively to the one-dimensional
interval factors gives the corresponding result for cubes:
for every Borel probability measure~$\nu$~on~$\Bbb{R}^m$,
the measure~$L_{[0,1]^d}\times\nu$~is spectral if and only if
~$\nu$~is spectral.
\\
\indent
The proofs use the Jorgensen--Pedersen criterion and equality in Bessel's inequality.
For suitable nonnegative quantities~$a_\lambda,b_\lambda$~and~$M_i=\sup_{\lambda\in A_i}b_\lambda$, the orthogonality properties give
$$1=\sum_i\sum_{\lambda\in A_i}a_\lambda b_\lambda
\leq\sum_iM_i\sum_{\lambda\in A_i}a_\lambda
\leq\sum_iM_i\leq1.$$
Equality throughout forces the relevant Fourier sums to attain their Bessel bounds, which gives completeness for the factors.
For convolution, compact support of~$\mu$~provides the analytic continuation needed in the proof, while~$\nu$~only requires continuity of its Fourier transform.
For products, the two frequency variables can be chosen independently, so no analyticity is needed.
For convolution with a cube, we remove one coordinate interval factor at a time and use continuity to extend a Fourier-sum identity from a dense set.
\\
\indent
The rest of the paper is organized as follows. In Subsection~\ref{sect.2.1}, we collect the Fourier-sum lemmas used in the proofs.
Subsection~\ref{sect.2.2} contains the convolution theorem and its applications.
In Subsection~\ref{sect.2.3}, we prove the general product characterization and derive the subgroup and cube factor results.

\section{\bf Fourier-sum lemmas and factor theorems\label{sect.2}}
\setcounter{equation}{0}
\subsection{Preliminary lemmas\label{sect.2.1}}
\begin{lem}\label{lem2.1}
Let~$\mu$~be a Borel probability measure on~$\R^d$.
If~$|\widehat\mu(k)|=1$~for every~$k\in\Bbb{Z}^d$, then~$supp\mu\subset a+\Bbb{Z}^d$~for some~$a\in[0,1)^d$.
\end{lem}
\begin{proof}
Let~$u_1,\cdots,u_d$~be the standard coordinate vectors of~$\R^d$.
For each~$j=1,\cdots,d$, put
$$z_j=\widehat\mu(u_j)=\int_{\R^d}e^{-2\pi ix_j}\,d\mu(x).$$
The hypothesis gives~$|z_j|=1$. We first show that~$e^{-2\pi ix_j}=z_j$~for~$\mu$-almost every~$x$.
Indeed,
$$|e^{-2\pi ix_j}-z_j|^2
=1-\overline{z_j}e^{-2\pi ix_j}-z_je^{2\pi ix_j}+|z_j|^2.$$
Since~$\int e^{2\pi ix_j}\,d\mu(x)=\overline{z_j}$, integration gives
$$\int_{\R^d}|e^{-2\pi ix_j}-z_j|^2\,d\mu(x)=1-|z_j|^2=0.$$
The integrand is nonnegative, so it vanishes~$\mu$-almost everywhere.
\\
\indent
As~$|z_j|=1$, choose~$a_j\in[0,1)$~such that~$z_j=e^{-2\pi ia_j}$.
Explicitly, if~$z_j=e^{i\theta_j}$, one may take
$$a_j=-\frac{\theta_j}{2\pi}-\left\lfloor-\frac{\theta_j}{2\pi}\right\rfloor.$$
This choice is unique in~$[0,1)$, since two such choices differ by an integer.
It follows that~$e^{-2\pi i(x_j-a_j)}=1$~for~$\mu$-almost every~$x$, and hence
$$\mu(E_j)=1,\qquad E_j=\{x\in\Bbb{R}^d:x_j\in a_j+\Z\}.$$
Let~$a=(a_1,\cdots,a_d)$. Then~$\bigcap_{j=1}^dE_j=a+\Bbb{Z}^d$, and
$$\mu\big(\Bbb{R}^d\setminus(a+\Bbb{Z}^d)\big)
\leq\sum_{j=1}^d\mu(E_j^c)=0.$$
Finally,~$a+\Bbb{Z}^d$~is closed. Every point outside this set has a neighborhood of~$\mu$-measure zero, so~$supp\mu\subset a+\Bbb{Z}^d$.
\end{proof}

The first two assertions below are the Jorgensen--Pedersen
criterion for arbitrary Borel probability measures;
see~\cite[Theorem~2.6]{LMW24}.
We also record the continuity of the associated Fourier sum
and its analytic extension in the compactly supported case.
\begin{lem}[\cite{JP98,LMW24}]\label{lem2.2}
Let~$\mu$~be a Borel probability measure on~$\Bbb{R}^d$, and let~$\Lambda\subset\Bbb{R}^d$~be countable. Put
$$Q_{\mu,\Lambda}(\xi)=\sum_{\lambda\in\Lambda}|\widehat\mu(\xi+\lambda)|^2.$$
\begin{itemize}
\item[(i)]~$\Lambda$~is an orthogonal set for~$\mu$~if and only if~$Q_{\mu,\Lambda}(\xi)\leq1$~for every~$\xi\in\Bbb{R}^d$.
\item[(ii)]~$\Lambda$~is a spectrum of~$\mu$~if and only if~$Q_{\mu,\Lambda}(\xi)=1$~for every~$\xi\in\Bbb{R}^d$.
\item[(iii)] If~$\Lambda$~is orthogonal for~$\mu$, then~$Q_{\mu,\Lambda}$~is continuous on~$\Bbb{R}^d$.
If, in addition,~$\mu$~is compactly supported, then~$Q_{\mu,\Lambda}$~extends to an entire function on~$\mathbb{C}^d$.
\end{itemize}
\end{lem}

\begin{lem}\label{lem2.3}
Let~$A=\bigcup_{i\in I}A_i$~be a countable set partitioned into nonempty blocks.
Suppose that~$a_\lambda,b_\lambda\geq0$, and put~$M_i=\sup_{\lambda\in A_i}b_\lambda$.
Assume that
$$\sum_{\lambda\in A_i}a_\lambda\leq1,\qquad
\sum_iM_i\leq1,\qquad
\sum_i\sum_{\lambda\in A_i}a_\lambda b_\lambda=1.$$
Then~$\sum_iM_i=1$~and, for every~$i\in I$~and~$\lambda\in A_i$,
\begin{equation}\label{(2.1)}
a_\lambda(M_i-b_\lambda)=0,\qquad
M_i\left(1-\sum_{\lambda\in A_i}a_\lambda\right)=0.
\end{equation}
In particular,~$a_\lambda>0$~implies~$b_\lambda=M_i$, and~$M_i>0$~implies~$\sum_{\lambda\in A_i}a_\lambda=1$.
\end{lem}
\begin{proof}
Set
\[
s_i:=\sum_{\lambda\in A_i}a_\lambda\leq1.
\]
By the assumptions,
\[
\begin{aligned}
0
&=1-\sum_i\sum_{\lambda\in A_i}a_\lambda b_\lambda\\
&=\left(1-\sum_iM_i\right)
  +\sum_iM_i(1-s_i)
  +\sum_i\sum_{\lambda\in A_i}
       a_\lambda(M_i-b_\lambda).
\end{aligned}
\]
All terms on the right-hand side are nonnegative, since
$\sum_iM_i\leq1$, $s_i\leq1$, and $b_\lambda\leq M_i$
for $\lambda\in A_i$.
Consequently,
\[
\sum_iM_i=1,\qquad
M_i(1-s_i)=0,\qquad
a_\lambda(M_i-b_\lambda)=0
\quad(\lambda\in A_i).
\]
Thus $a_\lambda>0$ implies $b_\lambda=M_i$, and
$M_i>0$ implies $s_i=1$, as required.
\end{proof}

\subsection{Convolution measures\label{sect.2.2}}
We now apply Lemmas~\ref{lem2.2} and~\ref{lem2.3} to convolution measures.
Under the subgroup assumption below, each coset block of a convolution spectrum is orthogonal for the first factor, while every choice of one representative from each block is orthogonal for the second factor.
Equality in the corresponding Bessel estimates gives completeness and the following characterization.

\begin{thm}\label{thm2.4}
Let~$\mu$~be a compactly supported Borel probability measure on~$\Bbb{R}^d$, and let~$\nu$~be a Borel probability measure on~$\Bbb{R}^d$.
Suppose that~$\mathcal{Z}(\widehat\mu)\cup\{0\}=T$~for an additive subgroup~$T\subset\Bbb{R}^d$.
Define
$$T^\perp=\{x\in\Bbb{R}^d:\langle t,x\rangle\in\Bbb{Z}\text{ for all }t\in T\}.$$
Then~$\mu*\nu$~is spectral if and only if~$T$~is a spectrum of~$\mu$, ~$\nu$~is spectral, and
$$supp\nu\subset a+T^\perp\qquad\text{for some }a\in\Bbb{R}^d.$$
In this case, every nonempty block~$\Lambda\cap(c+T)$~of any spectrum~$\Lambda$~of~$\mu*\nu$~equals~$c+T$, and every choice of one point from each nonempty block is a spectrum of~$\nu$.
Moreover,
\begin{equation}\label{(2.2)}
|\widehat\nu(\xi+t)|^2=|\widehat\nu(\xi)|^2\qquad(\xi\in\Bbb{R}^d,~t\in T).
\end{equation}
Under these equivalent conditions,~$\Gamma+T$~is a spectrum of~$\mu*\nu$~for every spectrum~$\Gamma$~of~$\nu$.
\end{thm}
\begin{proof}
We prove necessity in five steps. Let~$\Lambda$~be a spectrum of~$\mu*\nu$.

\noindent\textbf{Step 1: Orthogonality and the Fourier-sum bounds.}
Partition~$\Lambda$~modulo~$T$~as
$$\Lambda=\bigcup_{i\in I}\Lambda_i,\qquad
\Lambda_i=\Lambda\cap(c_i+T)\neq\emptyset,$$
where the cosets~$c_i+T$~are pairwise distinct.
Distinct elements of the same block differ by an element of~$T\setminus\{0\}$, so each~$\Lambda_i$~is orthogonal for~$\mu$.
If~$\lambda\in\Lambda_i$~and~$\lambda'\in\Lambda_j$~with~$i\neq j$, then~$\widehat\mu(\lambda-\lambda')\neq0$.
Orthogonality for~$\mu*\nu$~therefore gives~$\widehat\nu(\lambda-\lambda')=0$.
Thus every choice~$\lambda_i\in\Lambda_i$~gives an orthogonal set~$\{\lambda_i:i\in I\}$~for~$\nu$.
For~$\xi\in\Bbb{R}^d$, define
$$Q_i(\xi)=\sum_{\lambda\in\Lambda_i}|\widehat\mu(\xi+\lambda)|^2,
\qquad M_i(\xi)=\sup_{\lambda\in\Lambda_i}|\widehat\nu(\xi+\lambda)|^2.$$
Lemma~\ref{lem2.2} gives~$Q_i(\xi)\leq1$.
We claim that~$\sum_{i\in I}M_i(\xi)\leq1$.
Fix~$\xi\in\Bbb{R}^d$, a finite nonempty subset~$F\subset I$, and~$\varepsilon>0$.
For each~$i\in F$, choose~$\lambda_i^*\in\Lambda_i$~such that
$$M_i(\xi)<|\widehat\nu(\xi+\lambda_i^*)|^2+\frac{\varepsilon}{\#F}.$$
Since~$\{\lambda_i^*:i\in F\}$~is orthogonal for~$\nu$, we obtain
$$\sum_{i\in F}M_i(\xi)<\sum_{i\in F}|\widehat\nu(\xi+\lambda_i^*)|^2+\varepsilon\leq1+\varepsilon.$$
Letting~$\varepsilon\downarrow0$~and taking the supremum over all finite nonempty~$F$~proves the claim.
Since~$\Lambda$~is a spectrum of~$\mu*\nu$, Lemma~\ref{lem2.2} gives
\begin{equation}\label{(2.3)}
1=\sum_i\sum_{\lambda\in\Lambda_i}|\widehat\mu(\xi+\lambda)|^2|\widehat\nu(\xi+\lambda)|^2.
\end{equation}

\noindent\textbf{Step 2: A spectrum of~$\nu$.}
Applying Lemma~\ref{lem2.3} with~$a_\lambda=|\widehat\mu(\xi+\lambda)|^2$~and~$b_\lambda=|\widehat\nu(\xi+\lambda)|^2$, we get
\begin{equation}\label{(2.4)}
\sum_iM_i(\xi)=1,\qquad
|\widehat\mu(\xi+\lambda)|^2\big(M_i(\xi)-|\widehat\nu(\xi+\lambda)|^2\big)=0,
\end{equation}
and~$M_i(\xi)(1-Q_i(\xi))=0$.
Fix~$\lambda_0\in\Lambda_i$.
Since~$\mu$~is compactly supported,~$\widehat\mu$~is real analytic and is not identically zero.
Hence~$U=\{\xi:\widehat\mu(\xi+\lambda_0)\neq0\}$~is dense in~$\Bbb{R}^d$.
For any~$\lambda\in\Lambda_i$, equation~\eqref{(2.4)} gives
$$|\widehat\nu(\xi+\lambda)|^2\leq|\widehat\nu(\xi+\lambda_0)|^2\qquad(\xi\in U).$$
By continuity, this inequality holds for every~$\xi\in\Bbb{R}^d$.
Taking the supremum over~$\lambda$~and using~$\lambda_0\in\Lambda_i$, we obtain
\begin{equation}\label{(2.5)}
M_i(\xi)=|\widehat\nu(\xi+\lambda_0)|^2\qquad(\xi\in\Bbb{R}^d,~\lambda_0\in\Lambda_i).
\end{equation}
For every set~$\Gamma=\{\lambda_i:i\in I\}$, equations~\eqref{(2.4)}--\eqref{(2.5)} give
$$\sum_{\lambda\in\Gamma}|\widehat\nu(\xi+\lambda)|^2=1.$$
Thus~$\Gamma$~is a spectrum of~$\nu$~by Lemma~\ref{lem2.2}.

\noindent\textbf{Step 3: A spectrum of~$\mu$.}
Fix~$i\in I$~and~$\lambda_0\in\Lambda_i$.
By~\eqref{(2.5)},~$M_i(-\lambda_0)=1$.
Continuity gives~$M_i>0$~on a neighborhood of~$-\lambda_0$.
The identity~$M_i(1-Q_i)=0$~implies~$Q_i=1$~on this neighborhood.
Since~$\mu$~is compactly supported and~$\Lambda_i$~is orthogonal for~$\mu$, Lemma~\ref{lem2.2} implies that~$Q_i$~is real analytic on~$\Bbb{R}^d$.
The identity theorem yields~$Q_i\equiv1$, so~$\Lambda_i$~is a spectrum of~$\mu$.

\noindent\textbf{Step 4:~$T$~is a spectrum of~$\mu$.}
If~$u\in(c_i+T)\setminus\Lambda_i$, then
$$u-\lambda\in T\setminus\{0\}=\mathcal{Z}(\widehat\mu)\qquad(\lambda\in\Lambda_i).$$
Thus the unit vector~$e_u$~would be orthogonal to an orthonormal basis of~$L^2(\mu)$, a contradiction.
Hence~$\Lambda_i=c_i+T$.
Translation of a spectrum preserves spectrality, so~$T$~is a spectrum of~$\mu$.
Applying~\eqref{(2.5)} to~$c_i$~and~$c_i+t$~in a fixed block and translating~$\xi$~gives~\eqref{(2.2)}.

\noindent\textbf{Step 5: The support of~$\nu$.}
Taking~$\xi=0$~in~\eqref{(2.2)} gives~$|\widehat\nu(t)|=1$~for every~$t\in T$.
By the argument in Lemma~\ref{lem2.1}, for each~$t\in T$,
$$e^{-2\pi i\langle t,x\rangle}=\widehat\nu(t)\qquad\text{for }\nu\text{-almost every }x.$$
By continuity, the equality holds at every point of~$supp\nu$.
Fix~$a\in supp\nu$. Then
$$e^{-2\pi i\langle t,x-a\rangle}=1\qquad(t\in T,~x\in supp\nu).$$
Hence~$\langle t,x-a\rangle\in\Z$, and therefore~$supp\nu\subset a+T^\perp$.

For sufficiency, suppose that~$T$~is a spectrum of~$\mu$, ~$\nu$~is spectral, and~$supp\nu\subset a+T^\perp$.
Let~$\Gamma$~be any spectrum of~$\nu$.
For every~$t\in T$~and~$x\in supp\nu$, we have~$\langle t,x-a\rangle\in\Z$.
Consequently,
$$\widehat\nu(\xi+t)=e^{-2\pi i\langle t,a\rangle}\widehat\nu(\xi)\qquad(\xi\in\Bbb{R}^d,~t\in T).$$
In particular,~$|\widehat\nu|^2$~is~$T$-periodic and~$|\widehat\nu(t)|=1$~for~$t\in T$.
If~$\gamma,\gamma'\in\Gamma$~are distinct, then~$\widehat\nu(\gamma-\gamma')=0$, so~$\gamma-\gamma'\notin T$.
Thus~$\Gamma+T=\bigcup_{\gamma\in\Gamma}(\gamma+T)$.
For every~$\xi\in\Bbb{R}^d$,
\begin{align*}
\sum_{\lambda\in\Gamma+T}|\widehat{\mu*\nu}(\xi+\lambda)|^2
&=\sum_{\gamma\in\Gamma}\sum_{t\in T}|\widehat\mu(\xi+\gamma+t)|^2|\widehat\nu(\xi+\gamma+t)|^2\\
&=\sum_{\gamma\in\Gamma}|\widehat\nu(\xi+\gamma)|^2
  \sum_{t\in T}|\widehat\mu(\xi+\gamma+t)|^2\\
&=\sum_{\gamma\in\Gamma}|\widehat\nu(\xi+\gamma)|^2=1.
\end{align*}
Lemma~\ref{lem2.2} shows that~$\Gamma+T$~is a spectrum of~$\mu*\nu$.
\end{proof}

The proof of Theorem~\ref{thm2.4} also gives a characterization for general partitions whose blocks are orthogonal for~$\mu$~and for which every choice
of one representative from each block forms an
orthogonal set for~$\nu$.
\begin{cor}\label{cor2.5}
Let~$\mu$~be a compactly supported Borel probability measure on~$\Bbb{R}^d$, and let~$\nu$~be a Borel probability measure on~$\Bbb{R}^d$.
Let~$\Lambda=\bigcup_{i\in I}\Lambda_i$~be a nonempty countable set partitioned into nonempty blocks.
Suppose that each~$\Lambda_i$~is orthogonal for~$\mu$~and that, for every choice~$\lambda_i\in\Lambda_i$, the set~$\{\lambda_i:i\in I\}$~is orthogonal for~$\nu$.
Then~$\Lambda$~is a spectrum of~$\mu*\nu$~if and only if each~$\Lambda_i$~is a spectrum of~$\mu$~and every such transversal is a spectrum of~$\nu$.
In this case, for every~$i\in I$~and~$\xi\in\Bbb{R}^d$,
$$|\widehat\nu(\xi+\lambda)|^2=|\widehat\nu(\xi+\lambda')|^2\qquad(\lambda,\lambda'\in\Lambda_i).$$
\end{cor}
\begin{proof}
For necessity, Steps~1--3 of the proof of Theorem~\ref{thm2.4} apply without change.
For sufficiency, choose~$\lambda_i\in\Lambda_i$~for each~$i\in I$.
Fix~$i\in I$~and~$\lambda\in\Lambda_i$.
Both~$\{\lambda_j:j\in I\}$~and~$\{\lambda_j:j\neq i\}\cup\{\lambda\}$~are spectra of~$\nu$.
Hence
$$|\widehat\nu(\xi+\lambda_i)|^2+\sum_{j\neq i}|\widehat\nu(\xi+\lambda_j)|^2
=1=|\widehat\nu(\xi+\lambda)|^2+\sum_{j\neq i}|\widehat\nu(\xi+\lambda_j)|^2.$$
The common series converges, so cancellation gives
$$|\widehat\nu(\xi+\lambda)|^2=|\widehat\nu(\xi+\lambda_i)|^2\qquad(\lambda\in\Lambda_i).$$
Since each~$\Lambda_i$~is a spectrum of~$\mu$, we obtain
\begin{align*}
\sum_{\lambda\in\Lambda}|\widehat{\mu*\nu}(\xi+\lambda)|^2
&=\sum_{i\in I}|\widehat\nu(\xi+\lambda_i)|^2
  \sum_{\lambda\in\Lambda_i}|\widehat\mu(\xi+\lambda)|^2=\sum_{i\in I}|\widehat\nu(\xi+\lambda_i)|^2=1.
\end{align*}
Lemma~\ref{lem2.2} proves the assertion.
\end{proof}

The following corollary strengthens the convolution-factor theorem of Wu~\cite[Theorem~1.2]{W24} by removing the finite-support and Fourier-periodicity assumptions on the second factor.
In fact, only the first factor is required to have compact support.
\begin{cor}\label{cor2.6}
Let~$\mu$~be a compactly supported Borel probability measure on~$\Bbb{R}^d$, and let~$\nu$~be a Borel probability measure on~$\Bbb{R}^d$.
Put~$\eta=\mu*\nu$~and~$D=\mathcal{Z}(\widehat\mu)\setminus\mathcal{Z}(\widehat\nu)$.
Assume that
\begin{equation}\label{cond1}
(D-D)\cap\mathcal{Z}(\widehat\eta)\subset D.
\end{equation}
If~$\eta$~is spectral, then both~$\mu$~and~$\nu$~are spectral.
\end{cor}
\begin{proof}
Let~$\Lambda$~be a spectrum of~$\eta$.
Define a relation on~$\Lambda$~by
$$\lambda\sim\lambda'\quad\Longleftrightarrow\quad
\lambda=\lambda'\text{ or }\lambda-\lambda'\in D.$$
We first verify that~$\sim$~is an equivalence relation.
Reflexivity is immediate.
Since~$\widehat\mu(-\xi)=\overline{\widehat\mu(\xi)}$~and the same identity holds for~$\nu$, both zero sets are symmetric about the origin.
Thus~$D=-D$, which proves symmetry.
For transitivity, suppose that~$\lambda\sim\lambda'$~and~$\lambda'\sim\lambda''$.
If any two of these points coincide, the conclusion follows immediately.
Otherwise,
$$\lambda-\lambda'\in D,\qquad\lambda''-\lambda'\in D.$$
Since~$\Lambda$~is orthogonal for~$\eta$, we also have~$\lambda-\lambda''\in\mathcal{Z}(\widehat\eta)$.
Applying~\eqref{cond1} to~$\alpha=\lambda-\lambda'$~and~$\beta=\lambda''-\lambda'$~gives
$$\lambda-\lambda''=\alpha-\beta\in D.$$
This proves transitivity.
\\
\indent
Let~$\Lambda_i$, ~$i\in I$, be the equivalence classes. Then~$\Lambda=\bigcup_{i\in I}\Lambda_i$.
Distinct elements of~$\Lambda_i$~have difference in~$D\subset\mathcal{Z}(\widehat\mu)$, so~$\Lambda_i$~is orthogonal for~$\mu$.
If~$\lambda\in\Lambda_i$~and~$\lambda'\in\Lambda_j$~with~$i\neq j$, then~$\lambda-\lambda'\notin D$.
On the other hand,
$$\lambda-\lambda'\in\mathcal{Z}(\widehat\eta)=\mathcal{Z}(\widehat\mu)\cup\mathcal{Z}(\widehat\nu).$$
By the definition of~$D$, it follows that~$\lambda-\lambda'\in\mathcal{Z}(\widehat\nu)$.
Thus every transversal~$\{\lambda_i:i\in I\}$~is orthogonal for~$\nu$.
Corollary~\ref{cor2.5} now implies that both~$\mu$~and~$\nu$~are spectral.
\end{proof}

Taking~$\mu=L_{[0,1]}$~and~$T=\Z$~in Theorem~\ref{thm2.4}, we obtain the following one-dimensional characterization.
\begin{cor}\label{cor2.7}
Let~$\nu$~be a Borel probability measure on~$\Bbb{R}$.
Then~$L_{[0,1]}*\nu$~is spectral if and only if~$\nu$~is spectral and
$$supp\nu\subset a+\Bbb{Z}\qquad\text{for some }a\in\Bbb{R}.$$
In this case,~$\nu$~is a uniform probability measure on a finite set; more precisely,
$$\nu=\frac1{\#B}\sum_{b\in B}\delta_{a+b}$$
for some finite nonempty set~$B\subset\Bbb{Z}$.
Moreover, if~$\Gamma$~is any spectrum of~$\nu$, then~$\Gamma+\Bbb{Z}$~is a spectrum of~$L_{[0,1]}*\nu$.
\end{cor}
\begin{proof}
For~$\mu=L_{[0,1]}$, we have~$\mathcal{Z}(\widehat\mu)\cup\{0\}=\Bbb{Z}$, ~$\Bbb{Z}^\perp=\Bbb{Z}$, and~$\Bbb{Z}$~is a spectrum of~$\mu$.
Thus the equivalence and the assertion concerning~$\Gamma+\Bbb{Z}$~follow from Theorem~\ref{thm2.4}.
It remains to describe~$\nu$~when~$L_{[0,1]}*\nu$~is spectral.
Since~$supp\nu\subset a+\Bbb{Z}$, the measure~$\nu$~is discrete.
By~\cite[Theorem~1.2(2)]{LW06}, a discrete spectral probability measure has finite support and equal weights.
Hence
$$\nu=\frac1{\#B}\sum_{b\in B}\delta_{a+b}$$
for some finite nonempty set~$B\subset\Bbb{Z}$.
\end{proof}

The following example shows that the lattice-support conclusion of Corollary~\ref{cor2.7} fails in higher dimensions.
\begin{rem}
Put
$$v=(1,1/2),\qquad\nu=\tfrac12(\delta_0+\delta_v),\qquad
\Omega=[0,1]^2\cup\big(v+[0,1]^2\big).$$
The two squares have disjoint interiors, and~$L_{[0,1]^2}*\nu=\tfrac12\mathbf1_\Omega\,dx$.
Up to boundaries,~$\Omega$~is a fundamental domain for the lattice~$2\Bbb{Z}\times\Bbb{Z}$.
Consequently, its normalized Lebesgue measure has spectrum~$(\tfrac12\Bbb{Z})\times\Bbb{Z}$.
However,~$\{0,v\}$~is not contained in a translate of~$\Bbb{Z}^2$.
For~$d>2$, taking~$v=(1,1/2,0,\cdots,0)$~gives the same example with spectrum~$(\tfrac12\Bbb{Z})\times\Bbb{Z}^{d-1}$.
\end{rem}

For~$d\geq2$, one has
$$\mathcal{Z}(\widehat{L_{[0,1]^d}})=\bigcup_{j=1}^d
\{\xi\in\Bbb{R}^d:\xi_j\in\Bbb{Z}\setminus\{0\}\}.$$
Adjoining the origin does not produce an additive subgroup, so Theorem~\ref{thm2.4} does not apply directly.
Nevertheless, the factor-spectrality conclusion remains valid.
The key observation is that $\mathcal L_{[0,1]^d}$ is the convolution
of the uniform probability measures on the coordinate unit segments.
For each segment factor, we partition the first-coordinate projection
of a spectrum into integer cosets and retain the corresponding
full frequency sections.

\begin{thm}\label{thm2.8}
Let~$\nu$~be a Borel probability measure on~$\Bbb{R}^d$.
If~$L_{[0,1]^d}*\nu$~is spectral, then~$\nu$~is spectral.
\end{thm}
\begin{proof}
We first prove that~$\sigma*\eta$~being spectral implies that~$\eta$~is spectral, where
$$\sigma=L_{[0,1]}\times\delta_{0^{d-1}}$$
and~$\eta$~is any Borel probability measure on~$\R^d$.
Here~$0^k$~denotes the zero vector in~$\Bbb{R}^k$; when~$d=1$, we simply take~$\sigma=L_{[0,1]}$.
Write
$$h(t)=\widehat{L_{[0,1]}}(t)=\int_0^1e^{-2\pi itx}\,dx.$$
Then~$h(0)=1$, ~$\mathcal{Z}(h)=\Bbb{Z}\setminus\{0\}$, and
$$\widehat{\sigma*\eta}(\xi)=h(\xi_1)\widehat\eta(\xi),\qquad \xi=(\xi_1,\cdots,\xi_d)\in\Bbb{R}^d.$$
Let~$\Lambda$~be a spectrum of~$\sigma*\eta$.

\noindent\textbf{Step 1: Orthogonality of the sections.}
Let~$\pi_1(\lambda)=\lambda_1$~be projection onto the first coordinate.
Partition~$\pi_1(\Lambda)$~modulo~$\Bbb{Z}$~as
$$\pi_1(\Lambda)=\bigcup_{i\in I}A_i,\qquad
A_i=\pi_1(\Lambda)\cap(c_i+\Bbb{Z})\neq\emptyset,$$
where the cosets~$c_i+\Bbb{Z}$~are pairwise distinct.
For~$a\in A_i$, define~$S_{i,a}=\{\lambda\in\Lambda:\lambda_1=a\}$.
Each~$S_{i,a}$~is orthogonal for~$\eta$.
Indeed, for distinct~$\lambda,\lambda'\in S_{i,a}$,
$$0=\widehat{\sigma*\eta}(\lambda-\lambda')=h(0)\widehat\eta(\lambda-\lambda')=\widehat\eta(\lambda-\lambda').$$
If~$i\neq j$, ~$\lambda\in S_{i,a}$~and~$\lambda'\in S_{j,b}$, then~$a-b\notin\Bbb{Z}$~and~$h(a-b)\neq0$.
Again, orthogonality gives~$\widehat\eta(\lambda-\lambda')=0$.
Thus, for every choice~$a_i\in A_i$, the disjoint union~$\bigcup_{i\in I}S_{i,a_i}$~is orthogonal for~$\eta$.

\noindent\textbf{Step 2: Equality in the Bessel bounds.}
For~$\xi\in\Bbb{R}^d$, set
$$q_{i,a}(\xi)=\sum_{\lambda\in S_{i,a}}|\widehat\eta(\xi+\lambda)|^2,\qquad
M_i(\xi)=\sup_{a\in A_i}q_{i,a}(\xi).$$
Since~$S_{i,a}$~is orthogonal for~$\eta$, we have~$0\leq q_{i,a}(\xi)\leq1$.
Also,~$A_i\subset c_i+\Bbb{Z}$~is orthogonal for~$L_{[0,1]}$, so
$$s_i(\xi)=\sum_{a\in A_i}|h(\xi_1+a)|^2\leq1.$$
We claim that~$\sum_{i\in I}M_i(\xi)\leq1$.
Fix~$\xi$, a finite nonempty~$F\subset I$, and~$\varepsilon>0$.
Choose~$a_i\in A_i$, ~$i\in F$, with~$M_i(\xi)<q_{i,a_i}(\xi)+\varepsilon/\#F$.
The union~$\bigcup_{i\in F}S_{i,a_i}$~is orthogonal for~$\eta$, so
$$\sum_{i\in F}M_i(\xi)<\sum_{i\in F}q_{i,a_i}(\xi)+\varepsilon\leq1+\varepsilon.$$
Letting~$\varepsilon\downarrow0$~and taking the supremum over finite~$F$~proves the claim.
Since~$\Lambda$~is a spectrum of~$\sigma*\eta$,
$$1=\sum_{i\in I}\sum_{a\in A_i}|h(\xi_1+a)|^2q_{i,a}(\xi)
\leq\sum_{i\in I}M_i(\xi)s_i(\xi)\leq\sum_{i\in I}M_i(\xi)\leq1.$$
Lemma~\ref{lem2.3} therefore gives
$$\sum_{i\in I}M_i(\xi)=1,\qquad
|h(\xi_1+a)|^2\big(M_i(\xi)-q_{i,a}(\xi)\big)=0.$$

\noindent\textbf{Step 3: A spectrum of~$\eta$.}
Define
$$G=\R\setminus\bigcup_{a\in\pi_1(\Lambda)}(\Bbb{Z}-a).$$
The excluded set is countable, so~$G$~is dense in~$\R$.
If~$\xi_1\in G$, then~$h(\xi_1+a)\neq0$~for every~$a\in\pi_1(\Lambda)$.
Consequently,~$q_{i,a}(\xi)=M_i(\xi)$~for all~$i\in I$~and~$a\in A_i$.
Fix one~$a_i\in A_i$~for each~$i$, independently of~$\xi$, and put~$\Gamma=\bigcup_{i\in I}S_{i,a_i}$.
By Step~1,~$\Gamma$~is orthogonal for~$\eta$.
For~$\xi_1\in G$,
$$Q_{\eta,\Gamma}(\xi)=\sum_{\lambda\in\Gamma}|\widehat\eta(\xi+\lambda)|^2
=\sum_{i\in I}q_{i,a_i}(\xi)=\sum_{i\in I}M_i(\xi)=1.$$
Lemma~\ref{lem2.2} gives continuity of~$Q_{\eta,\Gamma}$.
Since~$G\times\R^{d-1}$~is dense,~$Q_{\eta,\Gamma}\equiv1$~on~$\R^d$.
Thus~$\Gamma$~is a spectrum of~$\eta$.

Finally, for~$j=1,\cdots,d$, put
$$\sigma_j=\delta_{0^{j-1}}\times L_{[0,1]}\times\delta_{0^{d-j}},$$
where zero-dimensional factors are omitted.
Then~$L_{[0,1]^d}=\sigma_1*\cdots*\sigma_d$.
The same argument, with the~$j$th coordinate in place of the first, shows that~$\sigma_j*\eta$~being spectral implies that~$\eta$~is spectral.
Starting from~$\sigma_1*\cdots*\sigma_d*\nu$, we remove these factors successively and conclude that~$\nu$~is spectral.
\end{proof}

\begin{rem}
The converse fails even when~$\nu=L_{[0,1]^d}$.
Indeed,~$\nu$~is spectral, whereas
$$L_{[0,1]^d}*\nu=\tau^{\times d},\qquad\tau=L_{[0,1]}*L_{[0,1]}.$$
By Corollary~\ref{cor2.7},~$\tau$~is not spectral.
Moreover,~$\mathcal{Z}(\widehat\tau)\cup\{0\}=\Bbb{Z}$.
For~$d\geq2$, if~$\tau^{\times d}$~were spectral, Theorem~\ref{thm2.10} below, applied to~$\tau\times\tau^{\times(d-1)}$, would imply that~$\tau$~is spectral, a contradiction.
\end{rem}

\subsection{Product measures\label{sect.2.3}}
We first establish a blockwise characterization for product measures, which is the product analogue of Corollary~\ref{cor2.5}.
\begin{thm}\label{thm2.9}
Let~$\mu$~and~$\nu$~be Borel probability measures on~$\Bbb{R}^{d_1}$~and~$\Bbb{R}^{d_2}$, respectively.
Let~$\Lambda\subset\Bbb{R}^{d_1}\times\Bbb{R}^{d_2}$~be a nonempty countable set, and let~$\pi_1(\lambda,\eta)=\lambda$~be the projection onto~$\Bbb{R}^{d_1}$.
Suppose that~$\pi_1(\Lambda)=\bigcup_{i\in I}A_i$~is a partition into nonempty blocks.
For~$\lambda\in A_i$, define
$$B_{i,\lambda}=\{\eta\in\Bbb{R}^{d_2}:(\lambda,\eta)\in\Lambda\}.$$
Assume that each~$A_i$~is orthogonal for~$\mu$~and that, for every choice~$\lambda_i\in A_i$, the sets~$B_{i,\lambda_i}$~are pairwise disjoint and their union is orthogonal for~$\nu$.
Then~$\Lambda$~is a spectrum of~$\mu\times\nu$~if and only if each~$A_i$~is a spectrum of~$\mu$~and, for every choice~$\lambda_i\in A_i$, the union~$\bigcup_{i\in I}B_{i,\lambda_i}$~is a spectrum of~$\nu$.
In this case, for every~$i\in I$~and~$\zeta\in\Bbb{R}^{d_2}$, the sum
$$\sum_{\eta\in B_{i,\lambda}}|\widehat\nu(\zeta+\eta)|^2$$
is independent of~$\lambda\in A_i$.
\end{thm}
\begin{proof}
For~$\xi\in\Bbb{R}^{d_1}$~and~$\zeta\in\Bbb{R}^{d_2}$, define
$$Q_i(\xi)=\sum_{\lambda\in A_i}|\widehat\mu(\xi+\lambda)|^2,\qquad
q_{i,\lambda}(\zeta)=\sum_{\eta\in B_{i,\lambda}}|\widehat\nu(\zeta+\eta)|^2,$$
and~$M_i(\zeta)=\sup_{\lambda\in A_i}q_{i,\lambda}(\zeta)$.
By the orthogonality assumptions and Lemma~\ref{lem2.2},
$$0\leq Q_i(\xi)\leq1,\qquad0\leq q_{i,\lambda}(\zeta)\leq1.$$
The finite-selection argument in Step~1 of Theorem~\ref{thm2.4}, applied to the disjoint unions of sections, also gives~$\sum_{i\in I}M_i(\zeta)\leq1$.
\\
\indent
For necessity, suppose that~$\Lambda$~is a spectrum of~$\mu\times\nu$.
Lemma~\ref{lem2.2} yields
$$1=\sum_{i\in I}\sum_{\lambda\in A_i}|\widehat\mu(\xi+\lambda)|^2q_{i,\lambda}(\zeta).$$
Applying Lemma~\ref{lem2.3}, we obtain~$\sum_{i\in I}M_i(\zeta)=1$~and
$$|\widehat\mu(\xi+\lambda)|^2\big(M_i(\zeta)-q_{i,\lambda}(\zeta)\big)=0,\qquad
M_i(\zeta)(1-Q_i(\xi))=0.$$
Fix~$i\in I$~and~$\lambda\in A_i$.
Taking~$\xi=-\lambda$~and using~$\widehat\mu(0)=1$, we get
$$q_{i,\lambda}(\zeta)=M_i(\zeta)\qquad(\zeta\in\Bbb{R}^{d_2}).$$
Thus the section sum is independent of~$\lambda\in A_i$.
For any choice~$\lambda_i\in A_i$, put~$\Gamma=\bigcup_{i\in I}B_{i,\lambda_i}$.
This union is disjoint, and
$$\sum_{\eta\in\Gamma}|\widehat\nu(\zeta+\eta)|^2
=\sum_{i\in I}q_{i,\lambda_i}(\zeta)=\sum_{i\in I}M_i(\zeta)=1.$$
Hence~$\Gamma$~is a spectrum of~$\nu$.
To prove completeness of~$A_i$, fix~$\lambda\in A_i$~and choose~$\eta\in B_{i,\lambda}$.
Such an~$\eta$~exists because~$A_i\subset\pi_1(\Lambda)$.
Then
$$M_i(-\eta)\geq q_{i,\lambda}(-\eta)\geq|\widehat\nu(0)|^2=1.$$
The identity~$M_i(-\eta)(1-Q_i(\xi))=0$~therefore gives~$Q_i(\xi)=1$~for all~$\xi\in\Bbb{R}^{d_1}$.
Thus~$A_i$~is a spectrum of~$\mu$.
\\
\indent
For sufficiency, assume that every~$A_i$~is a spectrum of~$\mu$~and every union obtained by choosing one section from each block is a spectrum of~$\nu$.
Fix a choice~$\lambda_i\in A_i$.
For a fixed~$i\in I$~and~$\lambda\in A_i$, both
$$\bigcup_{j\in I}B_{j,\lambda_j}\qquad\text{and}\qquad
B_{i,\lambda}\cup\bigcup_{j\neq i}B_{j,\lambda_j}$$
are spectra of~$\nu$, with disjoint unions.
Consequently,
$$q_{i,\lambda_i}(\zeta)+\sum_{j\neq i}q_{j,\lambda_j}(\zeta)
=1=q_{i,\lambda}(\zeta)+\sum_{j\neq i}q_{j,\lambda_j}(\zeta).$$
Cancellation gives~$q_{i,\lambda}(\zeta)=q_{i,\lambda_i}(\zeta)$.
Since~$Q_i\equiv1$, we have
\begin{align*}
\sum_{(\lambda,\eta)\in\Lambda}|\widehat{\mu\times\nu}(\xi+\lambda,\zeta+\eta)|^2
&=\sum_{i\in I}\sum_{\lambda\in A_i}|\widehat\mu(\xi+\lambda)|^2q_{i,\lambda}(\zeta)\\
&=\sum_{i\in I}q_{i,\lambda_i}(\zeta)Q_i(\xi)\\
&=\sum_{i\in I}q_{i,\lambda_i}(\zeta)=1.
\end{align*}
Lemma~\ref{lem2.2} shows that~$\Lambda$~is a spectrum of~$\mu\times\nu$.
\end{proof}

As an application, we recover the following result of~\cite{KLLL26} by a different method, based on the blockwise Fourier-sum characterization above.
\begin{thm}\label{thm2.10}
Let~$\mu$~and~$\nu$~be Borel probability measures on~$\Bbb{R}^{d_1}$~and~$\Bbb{R}^{d_2}$, respectively.
Suppose that~$\mathcal{Z}(\widehat\mu)\cup\{0\}=T$~for an additive subgroup~$T\subset\Bbb{R}^{d_1}$.
Then~$\mu\times\nu$~is spectral if and only if~$T$~is a spectrum of~$\mu$~and~$\nu$~is spectral.
\\
\indent
More precisely, let~$\Lambda$~be any spectrum of~$\mu\times\nu$, and let~$\pi_1(\lambda,\eta)=\lambda$~be the first-coordinate projection.
Partition~$\pi_1(\Lambda)$~modulo~$T$~as
$$\pi_1(\Lambda)=\bigcup_{i\in I}A_i,\qquad
A_i=\pi_1(\Lambda)\cap(c_i+T)\neq\emptyset,$$
where the cosets~$c_i+T$~are pairwise distinct, and define
$$B_{i,\lambda}=\{\eta\in\Bbb{R}^{d_2}:(\lambda,\eta)\in\Lambda\}\qquad(\lambda\in A_i).$$
Then~$A_i=c_i+T$~for every~$i\in I$. For every choice~$\lambda_i\in A_i$, the sets~$B_{i,\lambda_i}$~are pairwise disjoint and their union is a spectrum of~$\nu$.
\end{thm}
\begin{proof}
Let~$\Lambda$~be a spectrum of~$\mu\times\nu$.
If~$\lambda,\lambda'\in A_i$~are distinct, then~$\lambda-\lambda'\in T\setminus\{0\}=\mathcal{Z}(\widehat\mu)$.
Thus~$A_i$~is orthogonal for~$\mu$.
Orthogonality of~$\Lambda$~gives
$$\widehat\mu(\lambda-\lambda')\widehat\nu(\eta-\eta')=0
\qquad\text{for distinct }(\lambda,\eta),(\lambda',\eta')\in\Lambda.$$
Taking~$\lambda=\lambda'$~shows that each~$B_{i,\lambda}$~is orthogonal for~$\nu$.
If~$i\neq j$, ~$\lambda\in A_i$~and~$\lambda'\in A_j$, then~$\widehat\mu(\lambda-\lambda')\neq0$, so~$\widehat\nu(\eta-\eta')=0$.
In particular,~$\eta\neq\eta'$.
It follows that every choice~$\lambda_i\in A_i$~gives pairwise disjoint sections~$B_{i,\lambda_i}$~whose union is orthogonal for~$\nu$.
Theorem~\ref{thm2.9} implies that each~$A_i$~is a spectrum of~$\mu$~and every such union is a spectrum of~$\nu$.
If~$u\in(c_i+T)\setminus A_i$, then~$u-\lambda\in T\setminus\{0\}=\mathcal{Z}(\widehat\mu)$~for every~$\lambda\in A_i$.
Thus~$e_u$~would be orthogonal to the orthonormal basis~$E(A_i)$, a contradiction.
Hence~$A_i=c_i+T$, and translation of the spectrum shows that~$T$~is a spectrum of~$\mu$.
Conversely, if~$T$~and~$\Gamma$~are spectra of~$\mu$~and~$\nu$, respectively, then~$T\times\Gamma$~is a spectrum of~$\mu\times\nu$.
\end{proof}
\begin{cor}\label{cor2.11}
Let~$\nu$~be a Borel probability measure on~$\Bbb{R}^m$.
For every positive integer~$d$, the measure~$L_{[0,1]^d}\times\nu$~is spectral if and only if~$\nu$~is spectral.
\end{cor}
\begin{proof}
Since~$\mathcal{Z}(\widehat{L_{[0,1]}})\cup\{0\}=\Bbb{Z}$, Theorem~\ref{thm2.10} shows that~$L_{[0,1]}\times\sigma$~being spectral implies that~$\sigma$~is spectral.
Applying this implication successively to the~$d$~interval factors of~$L_{[0,1]^d}\times\nu$~gives the necessity.
Conversely, if~$\Gamma$~is a spectrum of~$\nu$, then~$\Bbb{Z}^d\times\Gamma$~is a spectrum of~$L_{[0,1]^d}\times\nu$~by the tensor-product construction.
\end{proof}

\end{document}